\documentclass[12pt, onesided, reqno]{amsart}

\usepackage{amssymb}
\usepackage{amsmath}
\usepackage{color}
\usepackage{enumerate}
\usepackage{graphicx}

\evensidemargin\oddsidemargin

\newtheorem{theorem}{Theorem}
\newtheorem{proposition}[theorem]{Proposition}
\newtheorem{corollary}[theorem]{Corollary}
\newtheorem{lemma}[theorem]{Lemma}

\theoremstyle{definition}

\newtheorem{remark}[theorem]{Remark}

\newcommand{\eqnsection}{
\renewcommand{\theequation}{\thesection.\arabic{equation}}
    \makeatletter
    \csname  @addtoreset\endcsname{equation}{section}
    \makeatother}
\eqnsection

\def\E{\mathbb{E}}

\def\R{\mathbb{R}}

\def\Pb{\mathbb{P}}

\def\F{\mathcal{F}}

\newcommand{\equi}{\mathop{\sim}\limits}
\def\={{\,\;\mathop{=}\limits^{\text{(law)}}\;\,}}

\def\qed{\hfill$\square$}

\allowdisplaybreaks

\makeatletter
\@namedef{subjclassname@2020}{%
  \textup{2020} Mathematics Subject Classification}
\makeatother

\begin{document}

\title[]{Parisian times for linear diffusions}
\author[Christophe Profeta]{Christophe Profeta}

\address{
Universit\'e Paris-Saclay, CNRS, Univ Evry, Laboratoire de Math\'ematiques et Mod\'elisation d'Evry, 91037, Evry-Courcouronnes, France.
 {\em Email} : {\tt christophe.profeta@univ-evry.fr}
  }



\begin{abstract} 
We compute the joint distribution of the first times a linear diffusion makes an excursion longer than some given duration above (resp. below) some fixed level.  In the literature, such stopping times have been introduced and studied in the framework of \emph{Parisian} barrier options, mainly in the case of Brownian motion with drift. We also exhibit several independence properties, and provide some formulae for the associated ruin probabilities.
\end{abstract}

\maketitle

\section{Statement of the main results}\label{sec:1}

\subsection{Introduction}

Let $X$ be a regular diffusion, without killing, taking values in some interval $I\subseteq\R$. In this note, we are interested in the following \emph{Parisian} stopping times:
$$\kappa_r^{(a,\pm)} = \inf\{t>r,\;  t-g_t^{(a)} \geq r\text{ and } \pm(X_t-a)\geq0\}$$
where $a\in  \text{Int}(I)$, the interior of $I$, $r>0$ and 
$$g_t^{(a)} =\sup\{s\leq t,\, X_s=a\}.$$

These stopping times are  the first times the process $X$ makes an excursion above $a$ (resp. below $a$) longer than $r>0$. They were introduced by Chesney, Jeanblanc \& Yor \cite{CJY} for the pricing of some barrier options, in which the owner loses the option if the underlying asset $X$ remains long enough below or above some fixed level. Several papers have been devoted to the valuation of such options, but in the case of diffusion processes, to our knowledge, only the example of Brownian motion with drift has been studied in details so far. In the case of two barriers, we refer the reader to Dassios \& Wu \cite{DW2} who  start from a semi-Markov model and use a limiting procedure, and to Anderluh \& van der Weide \cite{AW} and Labart \& Lelong \cite{LL} who use mainly martingales methods.\\

We shall not tackle here the problem of computing Parisian options, but rather study the distribution of the quadruplet $(\kappa_u^{(a, -)}, \kappa_v^{(b, +)},   X_{ \kappa_u^{(a,-)}}, X_{ \kappa_v^{(b,+)}})$ where $a< b$ and $u,v>0$. This distribution is always at the core of all the aforementioned papers. Our approach will essentially rely on diffusion and excursion theory. \\

The paper is divided as follows. In the remainder of Section \ref{sec:1}, we introduce some notation and state our main result, as well as several corollaries and consequences. Section \ref{sec:2} is devoted to some preliminary lemmas, which will be needed for the proof of the main theorem in Section \ref{sec:3}. Finally, several examples are given in Section \ref{sec:4}, including Brownian motion with drift and some Bessel processes.

\subsection{Notations}

We shall work in the framework of Pitman \& Yor \cite{PY}.  Let us denote by $\Pb_x$ and $\E_x$ the probability and expectation of $X$ when starting from $x\in I$, and $(\F_t, t\geq0)$ its natural filtration. We assume that the infinitesimal generator of $X$ has the form
$$\mathcal{G} = \frac{1}{2}\sigma^2(x) \frac{\text{d}^2}{\text{d} x^2} + \mu(x) \frac{\text{d}}{\text{d} x} $$ 
on some appropriate domain, with $\sigma$ a strictly positive and continuous function and $\mu$ a locally bounded function. Recall that $X$ may also be defined by its scale function $s$ and speed measure $m$ which are given by :
\begin{equation}\label{eq:sm}
s^\prime(x) = \exp\left(- \int^x \frac{2\mu(y)}{\sigma^2(y)}dy\right)\qquad \text{ and }\quad m(dx) = \frac{2}{s^\prime(x)\sigma^2(x)}dx.
\end{equation}

\noindent
For $a\in \text{Int}(I)$, we define the first hitting time 
$$T_a = \inf\{t>0,\, X_t =a\}.$$
Its Laplace transform is classically given for $\lambda>0$ by :
$$\E_x\left[e^{-\lambda T_a}\right] = \begin{cases}
\Phi_{\lambda,-}(x)/ \Phi_{\lambda, -}(a) & \text{ if }x\leq a,\\
\Phi_{\lambda, +}(x)/ \Phi_{\lambda, +}(a) & \text{ if }x>a,
\end{cases}$$
where $\Phi_{\lambda,-}$ and $\Phi_{\lambda,+}$ are two solutions (respectively increasing and decreasing) of the ordinary differential equation 
$\displaystyle\mathcal{G} u = \lambda u$, subject to some appropriate boundary conditions. Applying the Markov property and using the continuity of the paths of $X$, we also have for $a\leq x\leq b$ and $\lambda>0$ :
$$\E_x\left[e^{-\lambda T_a}1_{\{T_a<T_b\}}\right] = \frac{W_\lambda^{(b)}(x)}{W_\lambda^{(b)}(a)}
\quad\text{ and }\quad 
\E_x\left[e^{-\lambda T_b}1_{\{T_b<T_a\}}\right] = \frac{W_\lambda^{(a)}(x)}{W_\lambda^{(a)}(b) }
$$
where, for $c\in I$, the function $W_\lambda^{(c)}$ is defined by :
$$W_\lambda^{(c)}(x) = \Phi_{\lambda,+}(x)\Phi_{\lambda,-}(c)-\Phi_{\lambda,-}(x)\Phi_{\lambda,+}(c).$$

\noindent
Recall next from It\^o \& McKean \cite[Section 6.2]{IMK} that $T_a$ admits a continuous probability density function 
and that there exist two L\'evy measures $\nu_+^{(a)}$ and $\nu_-^{(a)}$ such that 
$$\nu_\pm^{(a)}(dt) = \pm\lim_{\varepsilon\downarrow0} \frac{\Pb_{a\pm\varepsilon}(T_a\in dt)}{s(a\pm\varepsilon)-s(a)}= \pm \frac{1}{s^\prime(a)} \frac{\text{d}}{\text{d}x}\Pb_x(T_a\in dt)_{\big|x=a^\pm}.$$
Equivalently, $\nu_+^{(a)}$ and $\nu_-^{(a)}$ may be defined by their (modified) Laplace transforms :
$$\int_{(0,+\infty]}(1-e^{-\lambda t}) \nu^{(a)}_{\pm} (dt) = \mp\frac{1}{s^\prime(a)} \frac{\Phi_{\lambda, \pm}^{\prime}(a)}{\Phi_{\lambda, \pm}(a)}.$$
Some care must be taken when dealing with transient diffusions, as in this case, the measure $\nu^{(a)}_{\pm}$ might have an atom at $t=+\infty$ given by :
$$\nu^{(a)}_{\pm}(\{+\infty\}) =\mp \frac{1}{s^\prime(a)} \lim_{\lambda\downarrow0}\frac{\Phi_{\lambda, \pm}^{\prime}(a)}{\Phi_{\lambda, \pm}(a)}.$$

\noindent
We also introduce the Green kernel of $X$ which is defined by the formula
$$G_\lambda (x,y) = \int_0^{+\infty} e^{-\lambda t}p(t; x,y) dt$$
where $p(t; x,y)$ denotes the transition density of $X$ with respect to the speed measure $m$, i.e.
$$\Pb_x(X_t \in dy) =  p(t; x,y) m( dy).$$
The Green kernel admits the classic representation 
$$G_\lambda (x,y) = \omega_\lambda^{-1} \Phi_{\lambda,-}(x) \Phi_{\lambda,+}(y),\qquad x\leq y,$$
where the Wronskian $\omega_\lambda$ is defined by 
$$\omega_\lambda = \frac{ \Phi_{\lambda, -}^{\prime}(x) \Phi_{\lambda, +}(x)- \Phi_{\lambda, -}(x) \Phi^{\prime}_{\lambda,+}(x)}{s^\prime(x)}$$
and is independent from $x$.\\

\noindent
We finally define the positive and negative meanders of lengths $u$  by :
$$\forall \Lambda_u \in \F_u,\qquad 
\begin{cases}
\displaystyle\Pb^{\uparrow, u}_{a}\left(\Lambda_u\right)= \lim_{\varepsilon\downarrow0} \Pb_{a+\varepsilon}\left(\Lambda_u | T_a>u\right)\\
\vspace{-.3cm}\\
\displaystyle\Pb^{\downarrow, u}_{a}\left(\Lambda_u\right) = \lim_{\varepsilon\downarrow0} \Pb_{a-\varepsilon}\left(\Lambda_u | T_a>u\right).
\end{cases}
$$
The occurrence of meanders in the forthcoming formulae is not surprising as the random variables $X_{ \kappa_u^{(a,-)}}$ and $ X_{ \kappa_v^{(b,+)}}$ are respectively a.s. smaller than $a$ and greater than $b$.\\

\subsection{Main results}

Let $\alpha$ and $\beta$ be two continuous and bounded functions from $I$ to $[0,+\infty)$. For $a\leq x \leq b$ and $\gamma ,\lambda, u,v>0$ let us define the modified Laplace transform of the quadruplet $( \kappa_u^{(a, -)},\kappa_v^{(b, +)},X_{ \kappa_u^{(a,-)}}, X_{ \kappa_v^{(b,+)}} )$  by :
$$
F_{\alpha,\beta,\gamma, \lambda}(x) =  \E_x\left[e^{- \gamma \kappa_u^{(a,-)} -\lambda \kappa_v^{(b,+)} } \alpha\big(X_{ \kappa_u^{(a,-)}}\big) \beta\big( X_{ \kappa_v^{(b,+)}}\big)1_{\{\kappa_v^{(b,+)}\leq \kappa_u^{(a,-)}<+\infty\}}\right]. 
$$
The main result of the paper is the computation of this function. To simplify the statement, let us set :
$$\Psi_{\lambda}^{(\pm)}(a,b,r) = \frac{\Phi_{\lambda, \pm}(a)}{\Phi_{\lambda, \pm}(b)}\pm \frac{W_{\lambda}^{(b)}(a)}{\omega_{\lambda}}  \int_r^{+\infty} e^{-\lambda t} \nu^{(b)}_\pm(dt). $$

\begin{theorem}\label{theo:1}
For $a\leq x \leq b$, the function $F_{\alpha,\beta,\gamma, \lambda}$ admits the expression :
\begin{equation}
\label{eq:Fx}F_{\alpha,\beta,\gamma, \lambda}(x) =\left( \frac{W_{\gamma+\lambda}^{(b)}(x)}{W_{\gamma+\lambda}^{(b)}(a)} \frac{1}{\Psi_{\gamma+\lambda}^{(-)}(b,a,u) }  + \frac{W_{\gamma+\lambda}^{(a)}(x)}{W_{\gamma+\lambda}^{(a)}(b)}  \right)F_{\alpha,\beta,\gamma, \lambda}(b) \end{equation}
where

\begin{multline*}
 F_{\alpha,\beta,\gamma, \lambda}(b)\\=e^{-(\lambda +\gamma) v}\frac{W_{\gamma+\lambda}^{(b)}(a)}{\omega_{\gamma+\lambda}} \frac{ \nu^{(b)}_+[v,+\infty]  \E_b\left[e^{-\gamma\kappa_u^{(a,-)}}\right]
}{\Psi_{\gamma+\lambda}^{(+)}(a,b,v) -\frac{1}{\Psi_{\gamma+\lambda}^{(-)}(b,a,u)}}  \E_a^{\downarrow,u}\left[\alpha \left(X_{u}\right) \right] \E_{b}^{\uparrow,v}\left[   \frac{\Phi_{\gamma,+}(X_v)}{\Phi_{\gamma,+}(b)}\beta\left( X_v\right)\right]
\end{multline*}
and 
$$
\E_b\left[e^{-\gamma \kappa_u^{(a,-)} }\right] =e^{-\gamma u}\frac{\Phi_{\gamma,+}(b)}{\Phi_{\gamma,+}(a)} \frac{\nu^{(a)}_-[u,+\infty]  }{\frac{1}{G_\gamma(a,a)} +\int_u^{+\infty} e^{-\gamma t} \nu_-^{(a)}(dt) }.$$
\end{theorem}

\noindent
One may observe several independence properties in this formula.

\begin{enumerate}
\item Letting $\gamma\downarrow0$, we deduce that on the set $\{\kappa_v^{(b,+)}\leq \kappa_u^{(a,-)}<+\infty\}$, the random variables $\kappa_v^{(b, +)}$, $ X_{\kappa_u^{(a,-)}}$ and $X_{\kappa_v^{(b,+)}}$ are mutually independent. 
\item Taking $\alpha\equiv1$, we observe that one may remove the condition $\{ \kappa_u^{(a,-)}<+\infty\}$ in the definition of $F_{1,\beta, \gamma, \lambda}$. Letting $\gamma\downarrow0$ and then  $u\rightarrow +\infty$, we deduce that on the set $\{\kappa_v^{(b,+)}<+\infty\}$, the random variables $\kappa_v^{(b, +)}$ and $X_{\kappa_v^{(b,+)}}$  are independent. This independence property, when dealing with a single boundary, was already observed in the Brownian case by Chesney, Jeanblanc \& Yor \cite{CJY}.\\
\end{enumerate}

We mention below two situations where Theorem \ref{theo:1} greatly simplifies. First, when $\alpha\equiv\beta\equiv1$, one obtains the following corollary :
\begin{corollary}\label{coro:1}
For $a<b$ and $ \gamma,\lambda>0$, the Laplace transform of the pair $(\kappa_u^{(a,-)} , \kappa_v^{(b,+)})$ is given by :
$$
\E_b\left[e^{-\gamma \kappa_u^{(a,-)} - \lambda \kappa_v^{(b,+)} }1_{\{\kappa_v^{(b,+)}\leq \kappa_u^{(a,-)}\}}\right] \\
=e^{-\lambda v}\frac{W_{\gamma+\lambda}^{(b)}(a)}{\omega_{\gamma+\lambda}} \frac{
\int_v^{+\infty}e^{-\gamma t}\nu^{(b)}_+(dt) }{\Psi_{\gamma+\lambda}^{(+)}(a,b,v) -\frac{1}{\Psi_{\gamma+\lambda}^{(-)}(b, a, u)}} \E_b\left[e^{-\gamma\kappa_u^{(a,-)}}  \right]. 
$$
\end{corollary}
\noindent
This result follows from the observation that 
\begin{align*}
\E_{b}^{\uparrow,v}\left[   \frac{\Phi_{\gamma,+}(X_v)}{\Phi_{\gamma,+}(b)}\right] & =\lim_{\varepsilon\downarrow 0}\frac{\E_{b+\varepsilon}\left[ \E_{X_v}\left[e^{-\gamma T_b}\right]1_{\{T_b \geq v\}}   \right]}{\Pb_{b+\varepsilon}(T_b\geq v) }\\
&=\lim_{\varepsilon\downarrow 0}\frac{\E_{b+\varepsilon}\left[ e^{- \gamma (T_b-v)}1_{\{T_b \geq v\}}   \right]}{\Pb_{b+\varepsilon}(T_b\geq v) } = \frac{e^{\gamma v}\int_v^{+\infty}  e^{-\gamma t} \nu_+^{(b)}(dt)}{\nu_+^{(b)}[v, +\infty]}
\end{align*}
where we have applied the Markov property in the second equality. Another simplification is obtained by letting $a\uparrow b$.  Using  that $W_{\lambda}^{(a)}(b)\xrightarrow[a\uparrow b]{}0$ and computing the limit yields the general formula for the one barrier case :

\begin{corollary}\label{coro:2}
For  $\gamma, \lambda>0$, the Laplace transform of the pair $(\kappa_u^{(b,-)} , \kappa_v^{(b,+)})$ is given by :
\begin{multline*}
\E_b\left[e^{-\gamma \kappa_u^{(b,-)}-\lambda \kappa_v^{(b,+)} }1_{\{\kappa_v^{(b,+)}\leq \kappa_u^{(b,-)}\}}\right] \\
\\= \frac{e^{-\lambda v}\int_v^{+\infty}  e^{-\gamma t} \nu^{(b)}_+(dt)}{
  \int_v^{+\infty} e^{-(\gamma+\lambda)t} \nu^{(b)}_+(dt) + \int_u^{+\infty}e^{-(\gamma+\lambda)t}\nu^{(b)}_-(dt) +\frac{1}{G_{\gamma+\lambda}(b,b)} }\E_b\left[e^{-\gamma\kappa_u^{(b,-)}}\right]. \end{multline*}
\end{corollary}
In particular, when $X$ is a Brownian motion with drift, by letting $\gamma \downarrow 0$, we recover the formula obtained by Dassios \& Wu \cite{DW1}.

\subsection{Parisian ruin probabilities}

Letting $\gamma\downarrow0$ in the Laplace transform of $\kappa^{(a,-)}_u$ given in Theorem \ref{theo:1}, we obtain the following ruin probabilities :
\begin{enumerate}
\item if $X$ is recurrent, then $\lim\limits_{\gamma\downarrow 0} G_{\gamma}(a,a)=+\infty$ and $\nu_-^{(a)}(\{+\infty\})=0$, hence
$$\Pb_a\left(\kappa_u^{(a,-)}<+\infty\right) = \frac{\nu^{(a)}_-[u,+\infty]}{\nu^{(a)}_-[u,+\infty)}=1.$$
\item if $X$ is transient, then $G_0(a,a)=\lim\limits_{\gamma\downarrow 0} G_{\gamma}(a,a)<+\infty$ and 
$$\Pb_a\left(\kappa_u^{(a,-)}<+\infty\right) = \frac{\nu^{(a)}_-[u,+\infty]}{\frac{1}{G_0(a,a)}+\nu^{(a)}_-[u,+\infty)}.$$
In particular, if $X$ goes to $\inf(I)$ a.s. as $t\rightarrow +\infty$, then $\frac{1}{G_0(a,a)} = \nu_-^{(a)}(\{+\infty\})$ and this probability also equals 1.
\end{enumerate}
Besides Brownian motion with drift, such Parisian ruin probabilities have also been computed for spectrally negative L\'evy processes, see for instance Loeffen, Czarna \& Palmowski \cite{LCZ}. \\

Another probability of interest is the probability that $\kappa_v^{(b,+)}$ takes place before $\kappa_u^{(a,-)}$, which may be obtained by letting $\lambda$ and $\gamma$ go to 0 in Corollary \ref{coro:1} and Formula (\ref{eq:Fx}).

\begin{corollary}\label{cor:kb<ka}
Assume that either $X$ is recurrent, or $X$ is transient and goes a.s to $\sup(I)$ or to $\inf(I)$ as $t\rightarrow +\infty$. Then, for $a\leq x\leq b$ :
\begin{multline*}
\Pb_x\left(\kappa_v^{(b,+)}\leq \kappa_u^{(a,-)} <+\infty \right) \\
= \left( \frac{s(b)-s(x)}{\frac{1}{\Pb_a(T_b<+\infty)} + (s(b)-s(a)) \nu^{(a)}_-[u,+\infty)  } +( s(x)-s(a))\right)\frac{\Pb_b\left(\kappa_v^{(b,+)}\leq \kappa_u^{(a,-)} <+\infty \right)}{s(b)-s(a)}
\end{multline*}
with
\begin{multline*}
\Pb_b\left(\kappa_v^{(b,+)}\leq \kappa_u^{(a,-)} <+\infty \right)\\
= \frac{(s(b)-s(a))  \nu^{(b)}_+[v,+\infty)\Pb_b(T_a<+\infty) \Pb_a\left(\kappa_u^{(a,-)}<+\infty\right)}{\frac{1}{\Pb_b(T_a<+\infty)} +(s(b)-s(a))  \nu^{(b)}_+[v,+\infty)  -\frac{1}{\frac{1}{\Pb_a(T_b<+\infty)} +(s(b)-s(a)) \nu^{(a)}_-[u,+\infty)    }}  .
\end{multline*}
\end{corollary}

\begin{remark}\label{rem:rec}
When $X$ is recurrent, since $\Pb_b(T_a<+\infty)=\Pb_a(T_b<+\infty)=1$, Corollary 4 greatly simplifies and one obtains :
$$\Pb_x\left(\kappa_v^{(b,+)}\leq \kappa_u^{(a,-)} \right) = \frac{\bigg(1+(s(x)-s(a))\nu^{(a)}_-[u,+\infty)\bigg) \nu^{(b)}_+[v,+\infty)}{\nu^{(a)}_-[u,+\infty)+\nu^{(b)}_+[v,+\infty) + (s(b)-s(a))\nu^{(a)}_-[u,+\infty)\nu^{(b)}_+[v,+\infty)}.$$
\end{remark}

\medskip
\noindent
To prove Corollary \ref{cor:kb<ka},  observe that by definition of the Wronskian $\omega_\lambda$, we have
\begin{align*}
\frac{W^{(b)}_\lambda(a)}{\omega_\lambda}  &= \Phi_{\lambda,+}(a)\Phi_{\lambda,+}(b)\int_a^b \frac{s^\prime(z)}{\Phi^2_{\lambda,+}(z)}dz = \int_a^b \frac{\E_b\left[e^{-\lambda T_z}\right]}{\E_z\left[e^{-\lambda T_a}\right]}s^\prime(z)dz. 
\end{align*}
Then, on the one hand, when $X$ is recurrent or $X$ is transient and goes a.s. to $\inf(I)$,  we obtain
$$\frac{W^{(b)}_\lambda(a)}{\omega_\lambda}  \xrightarrow[\lambda\downarrow0]{} \int_a^b s^\prime(z) dz = s(b)-s(a).$$
On the other hand, when $X$ is transient and goes a.s. to $\sup(I)$, we may assume without loss of generality that $s(\sup(I))=0$ and again
$$\frac{W^{(b)}_\lambda(a)}{\omega_\lambda}  \xrightarrow[\lambda\downarrow0]{} \int_a^b \frac{s(b)/s(z)}{s(z)/s(a)}s^\prime(z) dz = s(b)-s(a).$$
Corollary  \ref{cor:kb<ka} now follows by letting $\gamma$ and $\lambda$ go to 0 in Corollary \ref{coro:1} and Formula (\ref{eq:Fx}) and using the above limit.\\

\subsection{Remarks on the meanders}
We shall give in Section \ref{sec:4} a few examples where all the terms appearing in Theorem \ref{theo:1} may be computed explicitly.
To this end, it is worthwhile to note that the terms involving the meanders in Theorem \ref{theo:1} only require the knowledge of the transition density $p$ and of the L\'evy measures $\nu^{(a)}_{\pm}$. Indeed, by applying the Markov property
$$
\E_x\left[\alpha(X_u) 1_{\{T_a\geq u\}}\right]  =\E_x\left[\alpha(X_u)\right] - \int_0^u \E_a\left[\alpha(X_{u-t})\right]  \Pb_x(T_a\in dt).
$$
Inverting this transform, we obtain after some simple manipulations :
\begin{multline*}
\Pb_x\left(X_u\in dz |T_a\geq u\right)  / m(dz)\\
= \frac{p(u; x,z) - p(u; a,z)}{\Pb_x(T_a\geq u)} + p(u; a,z) + \int_0^u \frac{p(u; a,z) - p(u-t; a,z)}{\Pb_x(T_a\geq u)} \Pb_x(T_a\in dt).
\end{multline*}
Letting then $x\uparrow a$, we deduce that for $z<a$ :
\begin{multline*}
\Pb_a^{(\downarrow, u)}\left(X_u\in dz\right)/m(dz) \\
= - \frac{1}{s^\prime(a)\nu^{(a)}_-[u,+\infty]} \frac{\partial p}{\partial x}(u;x,z)_{|x=a}  + p(u; a,z) + \int_0^u \frac{p(u; a,z) - p(u-t; a,z)}{\nu^{(a)}_-[u,+\infty]} \nu_-^{(a)}(dt).
\end{multline*}
One may also give an alternative definition of the meanders as derivatives. Indeed, for example, by definition of the left derivative :
\begin{align}
\notag\Pb_a^{(\downarrow, u)}\left(\Lambda_u\right) &= \lim_{\varepsilon\downarrow0}  \frac{-\varepsilon}{\Pb_{a-\varepsilon}\left(T_a \geq  u\right)}\frac{ \Pb_{a-\varepsilon}\left(  \Lambda_u 1_{\{T_a\geq u \}} \right)}{-\varepsilon}\\
\label{eq:meander}&= -\frac{1}{s^\prime(a)\nu^{(a)}_-[u,+\infty] }\frac{\text{d} }{\text{d} x} \Pb_x\left(\Lambda_u1_{\{T_a\geq u \}} \right)_{\big|x=a^-}.
\end{align}
Formula (\ref{eq:meander}) will be the one appearing during the proofs.

\section{Preliminary computations}\label{sec:2}

Our proof will rely on some local estimates for $t\rightarrow p(t; a,a)$. Let us introduce the semimartingale local time of $X$ :
$$L_t^a = \sigma^2(a) \lim_{\varepsilon\downarrow 0} \frac{1}{2\varepsilon} \int_0^{t}1_{\{|X_s-a|\leq \varepsilon\}}ds,\qquad t\geq0.$$
From the occupation time formula (see \cite[Chapter VI]{RY}) and the expression (\ref{eq:sm}) of the speed measure $m$, we have
$$\E_a\left[L_t^a\right] = \frac{2}{s^\prime(a)} \int_0^{t} p(t; a,a) dt .$$
In particular, using Fatou's lemma, and the fact that $\sigma>0$, we deduce that
$$\frac{ 2}{s^\prime(a)}\liminf_{\varepsilon\downarrow{0}} \frac{1}{\varepsilon}\int_0^{\varepsilon} p(t; a,a) dt = \liminf_{\varepsilon\downarrow{0}} \E_a\left[\frac{L_\varepsilon^a-L_0^a}{\varepsilon}\right] \geq\E_a\left[ \liminf_{\varepsilon\downarrow{0}} \frac{L_\varepsilon^a-L_0^a}{\varepsilon}\right] =+\infty $$
which yields a first estimate :
\begin{equation}
\label{eq:e/p}\lim_{\varepsilon\downarrow{0}} \frac{\varepsilon}{\int_0^{\varepsilon} p(t; a,a) dt}=0.
\end{equation}

\noindent
We now prove a lemma which will be repeatedly used in the sequel.

\subsection{Limiting lemma}

\begin{lemma}\label{lem:1}
Let $a\in \text{Int}(I)$ and let $\varphi$ be a positive and continuous function such that $\E_x\left[\varphi(X_t)\right]<+\infty$ is finite for $t$ small enough and $x$ in the vicinity of $a$. We assume that $\varphi(a)=0$ and that $\varphi$ admits a left and a right derivative at the point $a$. Then the following asymptotics hold:
$$\E_a\left[\varphi(X_{\varepsilon}) 1_{\{X_{\varepsilon}\geq a\}}\right]\; \equi_{\varepsilon \downarrow 0}\; \frac{1}{s^\prime(a)}\varphi^\prime(a^+)\int_0^{\varepsilon} p(t; a,a) dt $$
and
$$\E_a\left[\varphi(X_{\varepsilon}) 1_{\{X_{\varepsilon}\leq a\}}\right]\; \equi_{\varepsilon \downarrow 0}\;-\frac{1}{s^\prime(a)}\varphi^\prime(a^-) \int_0^{\varepsilon} p(t; a,a) dt.$$
\end{lemma}

\begin{proof}
We only prove the first asymptotics, as the second one will follow from a similar argument. Take $\delta$ small enough and let us decompose the first expectation in Lemma \ref{lem:1} according as whether $X_\varepsilon$ is greater or smaller than $a+\delta$ : 
$$\E_a\left[\varphi(X_{\varepsilon}) 1_{\{X_{\varepsilon}\geq a\}}\right] = \E_a\left[\varphi(X_{\varepsilon}) 1_{\{X_{\varepsilon}\geq a+\delta\}}\right]+\E_a\left[\varphi(X_{\varepsilon}) 1_{\{a+\delta> X_{\varepsilon}\geq a\}}\right].$$
Applying the Markov property, the first term on the right-hand side is smaller than 
\begin{align*}
\E_a\left[\varphi(X_\varepsilon)1_{\{X_\varepsilon\geq a+\delta\}}\right] &=   \E_a\left[ \E_{a+\delta}\left[\varphi(X_{\varepsilon-s})1_{\{X_{\varepsilon-s}\geq a+\delta\}}\right]_{\big| s = T_{a+\delta}}1_{\{T_{a+\delta}\leq \varepsilon\}} \right]\\
&\leq   \sup_{s\in[0,\varepsilon]}\E_{a+\delta}\left[ \varphi(X_s)   \right] \Pb_a\left(T_{a+\delta}\leq \varepsilon\right)
\end{align*}
From Kotani \& Watanabe \cite{KW}, there is the asymptotics for $z\in \text{Int}(I)$ :
\begin{equation}\label{eq:KW}
\lim_{\varepsilon\downarrow 0} -2 \varepsilon \ln\left( \Pb_a\left(T_{z}\leq \varepsilon\right) \right) = \left(\int_a^{z}\frac{1}{\sigma(u)\sqrt{s^\prime(u)}}du \right)^2.
\end{equation}
 As a consequence, we deduce from (\ref{eq:e/p}) that 
$$
\lim_{\varepsilon\rightarrow 0} \frac{\E_a\left[\varphi(X_\varepsilon)1_{\{X_\varepsilon\geq a+\delta\}}\right] }{\int_0^\varepsilon p(t; a,a) dt} =0
$$
which proves that this term is negligible. Then, since $\varphi(a)=0$, the second term may be bounded by :
\begin{equation}\label{eq:infsup}
 \inf_{a \leq z\leq a+\delta}\frac{\varphi(z)}{z-a} \leq \frac{\E_a\left[\varphi(X_{\varepsilon}) 1_{\{a+\delta> X_{\varepsilon}\geq  a\}}
\right]}{ \E_a\left[(X_{\varepsilon}-a) 1_{\{a+\delta>X_{\varepsilon}\geq a\}} \right]} \leq  \sup_{a \leq z\leq a+\delta}\frac{\varphi(z)}{z-a}.
\end{equation}
Recall now that for any $x\in \text{Int}(I)$, the function $(t,y)\longrightarrow p(t; x,y)$ is a solution of the partial differential equation 
$\displaystyle \frac{\partial }{\partial t}p = \mathcal{G}p$ on $(0,+\infty)\times\text{Int}(I)$, 
and that for $z\neq a$, from Kotani \& Watanabe \cite{KW} :
\begin{equation}\label{eq:p0}
\lim_{t\downarrow0}p(t; a, z) = 0.
\end{equation}
Since the scale function $s$ and the speed measure $m$ are assumed to be absolutely continuous, an adaptation of their proof shows also that
\begin{equation}\label{eq:p'0}
\lim_{t\downarrow0}\frac{\partial p}{\partial z}(t; a, z) = 0.
\end{equation}
Now, integrating the partial differential equation for $p$ on $[0,\varepsilon]$ and using (\ref{eq:p0}), we deduce that for $z\neq a$, 
$$p(\varepsilon; a,z) =  \int_0^{\varepsilon}  \left( \frac{1}{2}\sigma^2(z) \frac{\partial^2 p}{\partial z^2} (t; a,z) + \mu(z)  \frac{\partial p}{\partial z} (t; a,z)\right) dt. $$
Then applying Fubini theorem, and recalling that $m(dx) = 2/(s^\prime(x)\sigma^2(x))dx$, we may write 
\begin{multline*}
\E_a\left[(X_{\varepsilon}-a) 1_{\{a+\delta>X_{\varepsilon}\geq a\}} \right]\\
= \int_0^{\varepsilon} \left(\int_a^{a+\delta} \frac{z-a}{s^\prime(z)}  \frac{\partial^2 p}{\partial z^2} (t; a,z) dz +  \int_a^{a+\delta}(z-a) \mu(z)  \frac{\partial p}{\partial z} (t; a,z) m(dz)\right) dt.
\end{multline*}
Integrating by parts the first term, and using the expression (\ref{eq:sm}) of  $s$, we observe that the terms in $\mu$ cancel and we obtain :
$$
\E_a\left[(X_{\varepsilon}-a) 1_{\{a+\delta>X_{\varepsilon}\geq a\}} \right]\\
=
\int_0^{\varepsilon}\left(\frac{\delta}{s^\prime(a+\delta)}\frac{\partial p}{\partial z} (t; a,a+\delta)   -  \int_a^{a+\delta} \frac{1}{s^\prime(z)}  \frac{\partial p}{\partial z} (t; a,z)  dz\right) dt.
$$
Using yet another integration by parts, we finally obtain :
\begin{multline*}
\E_a\left[(X_{\varepsilon}-a) 1_{\{a+\delta>X_{\varepsilon}\geq a\}} \right] =\frac{1}{s^\prime(a)}\int_0^{\varepsilon} p(t; a, a) dt\\
+\int_0^{\varepsilon} \left( \frac{1}{s^\prime(a+\delta)}\left(\delta\frac{\partial p}{\partial z} (t; a,a+\delta)  -p(t; a, a+\delta)\right) + \E_a\left[\mu(X_t)1_{\{a+\delta \geq X_t\geq a\}}\right]\right) dt.
\end{multline*}
As a consequence, going back to (\ref{eq:infsup}), we deduce from (\ref{eq:p0}), (\ref{eq:p'0}), and (\ref{eq:e/p}) that
\begin{multline*}
 \inf_{a \leq z\leq a+\delta}\frac{\varphi(z)}{z-a} \leq  \liminf_{\varepsilon\downarrow0} \frac{s^\prime(a)\E_a\left[\varphi(X_{\varepsilon}) 1_{\{ a+\delta>X_{\varepsilon}\geq a\}}\right]}{\int_0^{\varepsilon} p(t; a,a) dt}\\\leq 
 \limsup_{\varepsilon\downarrow0} \frac{s^\prime(a)\E_a\left[\varphi(X_{\varepsilon}) 1_{\{a+\delta> X_{\varepsilon}\geq a\}}\right]}{\int_0^{\varepsilon} p(t; a,a) dt}\leq  \sup_{a \leq z\leq a+\delta} \frac{\varphi(z)}{z-a} 
 \end{multline*}
 and the result follows by letting $\delta\downarrow0$.

\end{proof}

\subsection{The distribution of $(\kappa^{(a,-)}_u, X_{\kappa^{(a,-)}_u})$}

We now prove a simpler version of Theorem \ref{theo:1}, involving only $\kappa^{(a,-)}_u$ and  $X_{\kappa^{(a,-)}_u}$.

\begin{proposition}\label{prop:1}
On the set $\{\kappa^{(a,-)}_u<+\infty\}$ the random variables $\kappa^{(a,-)}_u$ and  $X_{\kappa^{(a,-)}_u}$ are independent, and their distributions are respectively given by
$$ \E_a\left[e^{-\gamma \kappa_u^{(a,-)}}\right] =e^{-\gamma u}\frac{\nu^{(a)}_-[u,+\infty]}{\frac{1}{G_\gamma(a,a)}+\int_u^{+\infty}e^{-\gamma t}\nu^{(a)}_-(dt)} $$
and
$$ \E_a\left[\alpha\big(X_{\kappa_u^{(a,-)}}\big)1_{\{\kappa_u^{(a,-)}<+\infty\}}\right] = \E_a^{\downarrow,u}\left[\alpha \left(X_{u} \right)\right].
 $$
\end{proposition}

\begin{proof}
We set
$$d_t^{(a)}=\inf\{s> t,\, X_s=a\}$$
and
$$F_{\alpha,\gamma}(a) = \E_a\left[e^{-\gamma \kappa_u^{(a,-)}}\alpha \big(X_{\kappa_u^{(a,-)}} \big)1_{\{\kappa_u^{(a,-)}<+\infty\}}\right].$$
Let $0<\varepsilon<u$. We shall decompose the expectation of $F_{\alpha,\gamma}(a)$ according as whether $X_{\varepsilon}$ is greater or smaller than $a$. When $X_{\varepsilon}>a$, applying the Markov property at the time $d_\varepsilon^{(a)}$, we obtain 
\begin{equation}
\label{eq:X>a}\E_a\left[e^{-\gamma \kappa_u^{(a,-)}}\alpha \big(X_{\kappa_u^{(a,-)}} \big)1_{\{\kappa_u^{(a,-)}<+\infty\}}1_{\{X_\varepsilon>a\}}\right]=\E_a\left[e^{-\gamma d_\varepsilon^{(a)}}1_{\{X_\varepsilon>a\}}\right]F_{\alpha,\gamma}(a).
\end{equation}
When $X_{\varepsilon}\leq a$, we need to  take care as whether or not the excursion below $a$ straddling $\varepsilon$ is longer than $u$. On the one hand, if this excursion is shorter than $u$, then
\begin{multline}
\label{eq:da<ga}
\E_a\left[e^{-\gamma \kappa_u^{(a,-)}}\alpha \big(X_{\kappa_u^{(a,-)}} \big)1_{\{\kappa_u^{(a,-)}<+\infty\}}1_{\{X_\varepsilon\leq a\}} 1_{\{d_\varepsilon^{(a)} < g_\varepsilon^{(a)}+u\}}\right]\\ = \E_a\left[e^{-\gamma d_\varepsilon^{(a)}}1_{\{X_\varepsilon\leq a\}}1_{\{d_\varepsilon^{(a)} < g_\varepsilon^{(a)}+u\}}\right]F_{\alpha,\gamma}(a).
\end{multline}
On the other hand, if this excursion is longer than $u$, then $\kappa_u^{(a,-)}  = g_\varepsilon^{(a)}+u$ a.s. and 
\begin{multline}
\label{eq:da>ga}\E_a\left[e^{-\gamma \kappa_u^{(a,-)}}\alpha \big(X_{\kappa_u^{(a,-)}} \big)1_{\{\kappa_u^{(a,-)}<+\infty\}}1_{\{X_\varepsilon\leq a\}} 1_{\{d_\varepsilon^{(a)} \geq  g_\varepsilon^{(a)}+u\}}\right] \\
 =\E_a\left[e^{-\gamma(g_\varepsilon^{(a)}+u)}  \alpha\big( X_{g_\varepsilon^{(a)}+u} \big)1_{\{X_\varepsilon\leq a\}} 1_{\{d_\varepsilon^{(a)} \geq  g_\varepsilon^{(a)}+u\}}\right] .
\end{multline}
Gathering (\ref{eq:X>a}), (\ref{eq:da<ga}) and (\ref{eq:da>ga}) together, we deduce that 
$$F_{\alpha,\gamma}(a) = \frac{\E_a\left[e^{-\gamma(g_\varepsilon^{(a)}+u)} \alpha \big( X_{g_\varepsilon^{(a)}+u} \big)1_{\{X_\varepsilon\leq a\}} 1_{\{d_\varepsilon^{(a)} \geq  g_\varepsilon^{(a)}+u\}}\right] }{1-\E_a\left[e^{-\gamma d_\varepsilon^{(a)}}1_{\{X_\varepsilon>a\}}\right]- \E_a\left[e^{-\gamma d_\varepsilon^{(a)}}1_{\{X_\varepsilon\leq a\}}1_{\{d_\varepsilon^{(a)} < g_\varepsilon^{(a)}+u\}}\right]}
$$
and it remains to let $\varepsilon\downarrow 0$. We start with the numerator. Fix $\delta \in (\varepsilon,u)$. Since $g_\varepsilon^{(a)}\leq \varepsilon$, we have 
\begin{align*}
&\E_a\left[e^{-\gamma g_\varepsilon^{(a)}} \alpha \big(X_{g_\varepsilon^{(a)}+u} \big)1_{\{X_\varepsilon\leq a\}} 1_{\{d_\varepsilon^{(a)} \geq  g_\varepsilon^{(a)}+u\}}\right] \\
&\qquad\qquad\qquad\qquad\qquad\leq \E_a\left[ \max\limits_{s\in[0,\delta]} \alpha\left( X_{s+u} \right)1_{\{X_\varepsilon\leq a\}} 1_{\{d_\varepsilon^{(a)} \geq  u\}}\right]\\
&\qquad\qquad\qquad\qquad\qquad= \E_a\left[  1_{\{X_\varepsilon\leq a\}}    \E_{X_\varepsilon}\left[\max\limits_{s\in[0,\delta]} \alpha\left( X_{s+u-\varepsilon} \right)  1_{\{\varepsilon+T_a \geq  u\}}\right]\right]\\
&\qquad\qquad\qquad\qquad\qquad\leq \E_a\left[  1_{\{X_\varepsilon\leq a\}}    \E_{X_\varepsilon}\left[   \E_{X_{u-\delta}}\left[\max\limits_{s\in[0,\delta]} \alpha\left( X_{s} \right)\right]   1_{\{T_a \geq  u-\delta\}}\right]\right].
\end{align*}
From Lemma \ref{lem:1} and Formula (\ref{eq:meander}) we deduce that :
\begin{multline*}
 \E_a\left[  1_{\{X_\varepsilon\leq a\}}    \E_{X_\varepsilon}\left[   \E_{X_{u-\delta}}\left[\max\limits_{s\in[0,\delta]} \alpha\left( X_{s} \right)\right]   1_{\{T_a \geq  u-\delta\}}\right]\right]\\
 \; \equi_{\varepsilon \downarrow 0}\;  \E_{a}^{\downarrow, u-\delta}\left[   \E_{X_{u-\delta}}\left[\max\limits_{s\in[0,\delta]} \alpha\left( X_{s} \right)\right] \right]\nu^{(a)}_-[u-\delta,+\infty]
  \int_0^{\varepsilon} p(t; a,a) dt. 
\end{multline*}
As a consequence, we obtain the upper bound  :
$$
\limsup_{\varepsilon\downarrow0} \frac{\E_a\left[e^{-\gamma g_\varepsilon^{(a)}} \alpha \big(X_{g_\varepsilon^{(a)}+u} \big)1_{\{X_\varepsilon\leq a\}} 1_{\{d_\varepsilon^{(a)} \geq  g_\varepsilon^{(a)}+u\}}\right]}{  \nu^{(a)}_-[u-\delta,+\infty] \int_0^{\varepsilon} p(t; a,a) dt} \\
\leq \E_a^{\downarrow,u-\delta}\left[\E_{X_{u-\delta}}\left[\max\limits_{s\in[0,\delta]} \alpha\left( X_{s} \right)\right] \right].
$$
Similarly, we have the lower bound :
$$
\liminf_{\varepsilon\downarrow0} \frac{\E_a\left[e^{-\gamma g_\varepsilon^{(a)}} \alpha \big(X_{g_\varepsilon^{(a)}+u} \big)1_{\{X_\varepsilon\leq a\}} 1_{\{d_\varepsilon^{(a)} \geq  g_\varepsilon^{(a)}+u\}}\right]}{ \nu^{(a)}_-[u,+\infty]  \int_0^{\varepsilon} p(t; a,a) dt} \\\geq e^{-\lambda \delta} \E_a^{\downarrow,u}\left[\E_{X_u}\left[\min\limits_{s\in[0,\delta]} \alpha\left( X_{s} \right)\right] \right].
$$
Letting $\delta\downarrow 0$, we deduce by continuity  that the numerator is equivalent to :
$$
 \E_a\left[e^{-\gamma g_\varepsilon^{(a)}} \alpha \big(X_{g_\varepsilon^{(a)}+u} \big)1_{\{X_\varepsilon\leq a\}} 1_{\{d_\varepsilon^{(a)} \geq  g_\varepsilon^{(a)}+u\}}\right]
\; \equi_{\varepsilon \downarrow 0}\;
 \E_a^{\downarrow,u}\left[\alpha\left( X_{u} \right) \right] \nu^{(a)}_-[u,+\infty]  \int_0^{\varepsilon} p(t; a,a) dt.
$$

\noindent
Looking now at the denominator, we obtain, adding and subtracting indicators functions :
\begin{multline*}
 \frac{1}{s^\prime(a)}\left( \frac{\text{d}}{\text{d}x}\E_{x}\left[e^{-\gamma T_a }1_{\{T_a \leq  u\}}\right]_{\big|x=a^-}  -  \frac{\text{d}}{\text{d}x}\E_{x}\left[e^{-\gamma T_a }\right]_{\big|x=a^+}\right) \int_0^{\varepsilon} p(t; a,a) dt. \\
 = \left(\frac{1}{G_\gamma(a,a)} +\int_u^{+\infty} e^{-\gamma t} \nu_-^{(a)}(dt)    \right) \int_0^{\varepsilon} p(t; a,a) dt.
 \end{multline*}
Gathering the two previous asymptotics, we finally obtain
$$\E_a\left[e^{-\gamma \kappa_u^{(a,-)}}\alpha \big(X_{\kappa_u^{(a,-)}} \big)1_{\{\kappa_u^{(a,-)}<+\infty\}}\right] = e^{-\gamma u}\frac{  \nu^{(a)}_-[u,+\infty] }{\frac{1}{G_\gamma(a,a)} +\int_u^{+\infty} e^{-\gamma t} \nu_-^{(a)}(dt) }\E_a^{\downarrow,u}\left[\alpha \left(X_{u} \right)\right] $$
which ends the proof of Proposition \ref{prop:1}.
\end{proof}

\begin{remark}
We briefly sketch here another proof for the distribution of $\kappa_u^{(a,-)}$ which is inspired from Getoor \cite{Ge}. Assume that $X$ is recurrent. The right-continuous inverse $(\tau_\ell^{a}, \, \ell\geq0)$ of $L^{a}$ is a subordinator with Laplace transform :
$$\E_a\left[e^{-\gamma \tau_\ell^a}\right] = \exp\left(- \ell\frac{s^\prime(a)}{G_\gamma(a,a)}\right) = \exp\left(-\ell s^\prime(a) \int_0^{+\infty} (1-e^{-\gamma t})\nu^{(a)}(dt)\right).$$
where $\nu^{(a)} = \nu^{(a)}_- + \nu^{(a)}_+$.  Furthermore the two processes
\begin{equation}\label{eq:n+n-}
\left(\int_0^{\tau_\ell} 1_{\{X_u>a\}} du, \ell\geq0\right) \quad\text{ and }\quad \left(\int_0^{\tau_\ell} 1_{\{X_u<a\}} du, \ell\geq0\right)
\end{equation}
are independent subordinators with respective L\'evy measures given by $s^\prime(a)\nu^{(a)}_+$ and $s^\prime(a)\nu^{(a)}_-$. As a consequence, we may decompose $\tau^{(a)}$ into two independent subordinators
$$\tau^{(a)} = \eta^{(0)} + \eta^{(1)}$$
with respective L\'evy measures $s^\prime(a)(\nu^{(a)}_+(dt) + 1_{\{t< u\}}\nu^{(a)}_-(dt))$ and $s^\prime(a)1_{\{t\geq  u\}}\nu^{(a)}_-(dt)$.
Now
$\kappa_u^{(a,-)}$ is the starting point of the first excursion below $a$ of length greater than $u$. 
Define the first jump of  $ \eta^{(1)}$ by
$$ R= \inf\{\ell>0, \; \eta^{(1)}_\ell\neq0\}.$$
The random variable $R$ is exponentially distributed with parameter $s^\prime(a)\nu^{(a,-)}[u,+\infty]$. Then
\begin{align*}
\E_a\left[e^{-\gamma (\kappa_u^{(a,-)}-u)}\right] &= \E\left[e^{-\gamma \eta^{(0)}_{R-}}\right]\\
&=s^\prime(a)\nu^{(a)}_-[u,+\infty]
 \int_0^{+\infty} \E\left[e^{-\gamma \eta^{(0)}_{r-}}\right] 
e^{- s^\prime(a)\nu^{(a)}_-[u,+\infty] r} dr\\
 &=\frac{\nu^{(a)}_-[u,+\infty]}{\int_0^{+\infty}(1-e^{-\gamma t})\nu^{(a)}_+(dt)+ \int_0^{u}(1-e^{-\gamma t})\nu^{(a)}_-(dt)+ 
 \nu^{(a)}_-[u,+\infty]}\\
 &=\frac{\nu^{(a)}_-[u,+\infty]}{\frac{1}{G_\gamma(a,a)} +\int_u^{+\infty} e^{-\gamma t} \nu_-^{(a)}(dt)},
\end{align*}
which yields the Laplace transform of $\kappa_u^{(a,-)}$  when $X$ is recurrent. However, when $X$ is transient, the two subordinators given in Formula (\ref{eq:n+n-}) are no longer independent as explained in \cite{PY}. They are only conditionally independent given $L_{\infty}^{a}$, since when $\ell = L_{\infty}^{a}$, one will be stopped at the time $L_{\infty}^{a}$ while the other one will jump to $+\infty$. As a consequence, in the transient case, one would need to condition by $L_{\infty}^{a}$ before mimicking the previous argument.
\end{remark}

\section{Proof of Theorem \ref{theo:1}}\label{sec:3}

The proof of Theorem \ref{theo:1} is similar to the proof of Proposition \ref{prop:1}, except that we shall work with two boundaries $a<b$, and thus write down a system of equations involving $F_{\alpha, \beta, \gamma,\lambda }(a)$ and $F_{\alpha, \beta, \gamma,\lambda }(b)$.
Let $0<\varepsilon< \min(u,v)$. To compute $F_{\alpha, \beta, \gamma,\lambda }(a)$ and $F_{\alpha, \beta, \gamma,\lambda }(b)$, we shall decompose the expectation in three terms, according to the position of $X_\varepsilon$ with respect to $a$ and $b$. 

\subsection{Study of $F_{\alpha, \beta, \gamma,\lambda }(a)$}
First, if $X_\varepsilon\leq a$, then the excursion below $a$ straddling $\varepsilon$ must remain shorter than $u$. Applying the Markov property at the time $d_\varepsilon^{(a)}$, this yields the formula
\begin{multline}
\E_a\left[e^{-\gamma \kappa_u^{(a,-)}- \lambda \kappa_v^{(b,+)}}   \alpha \big(X_{ \kappa_u^{(a,-)}}\big) \beta \big(X_{ \kappa_v^{(b,+)}}\big)1_{\{\kappa_v^{(b,+)}\leq \kappa_u^{(a,-)}<+\infty\}}1_{\{X_\varepsilon\leq a\}}\right]\\
\label{Faa}= \E_a\left[e^{-(\gamma+\lambda) d_\varepsilon^{(a)}}1_{\{d_\varepsilon^{(a)}< g_\varepsilon^{(a)}+u\}}1_{\{X_\varepsilon\leq a\}}\right]F_{\alpha, \beta, \gamma,\lambda }(a).
\end{multline}

\noindent
Next, if $a<X_\varepsilon<b$, we obtain, applying the Markov property at the times $d_\varepsilon^{(a)}$ and $d_\varepsilon^{(b)}$ :
\begin{multline}
\label{Faab}\E_a\left[e^{-\gamma \kappa_u^{(a,-)}- \lambda \kappa_v^{(b,+)} } \alpha \big(X_{ \kappa_u^{(a,-)}}\big) \beta \big(X_{ \kappa_v^{(b,+)}}\big)1_{\{\kappa_v^{(b,+)}\leq \kappa_u^{(a,-)}<+\infty\}}1_{\{a<X_\varepsilon< b\}}\right]\\=\E_a\left[e^{-(\gamma+\lambda) d_\varepsilon^{(b)}}1_{\{a<X_\varepsilon< b\}}1_{\{d_\varepsilon^{(b)}<d_\varepsilon^{(a)}\}}\right]F_{\alpha, \beta, \gamma,\lambda }(b)\\ +\E_a\left[e^{-(\gamma+\lambda) d_\varepsilon^{(a)}}1_{\{a<X_\varepsilon< b\}}1_{\{d_\varepsilon^{(a)}<d_\varepsilon^{(b)}\}}\right]F_{\alpha, \beta, \gamma,\lambda }(a).
\end{multline}

\noindent
The third and last term will have no contribution in the result :
\begin{equation}
\label{eq:Deltaa}\Delta_a(\varepsilon)=\E_a\left[e^{-\gamma \kappa_u^{(a,-)} - \lambda \kappa_v^{(b,+)}} \alpha \big(X_{ \kappa_u^{(a,-)}}\big) \beta \big(X_{ \kappa_v^{(b,+)}}\big)1_{\{\kappa_v^{(b,+)}\leq \kappa_u^{(a,-)}<+\infty\}}1_{\{X_\varepsilon\geq b\}}\right].
\end{equation}
\noindent
As a consequence, gathering (\ref{Faa}), (\ref{Faab}) and (\ref{eq:Deltaa}), we deduce that :
$$F_{\alpha, \beta, \gamma,\lambda }(a) = \frac{\E_a\left[e^{-(\gamma+\lambda) d_\varepsilon^{(b)}}1_{\{a<X_\varepsilon< b\}}1_{\{d_\varepsilon^{(b)}<d_\varepsilon^{(a)}\}}\right]F_{\alpha, \beta, \gamma,\lambda }(b)+\Delta_a(\varepsilon)}{1- \E_a\left[e^{-(\gamma+\lambda) d_\varepsilon^{(a)}}1_{\{d_\varepsilon^{(a)}< g_\varepsilon^{(a)}+u\}}1_{\{X_\varepsilon\leq a\}}\right]- \E_a\left[e^{-(\gamma+\lambda) d_\varepsilon^{(a)}}1_{\{a<X_\varepsilon< b\}}1_{\{d_\varepsilon^{(a)}<d_\varepsilon^{(b)}\}}\right]}$$
and it remains to let $\varepsilon\downarrow0$. Applying the Markov property and Lemma \ref{lem:1}, we deduce that the first term in the numerator is equivalent to :
\begin{align*}
\E_a\left[e^{-(\gamma+\lambda) d_\varepsilon^{(b)}}1_{\{a<X_\varepsilon< b\}}1_{\{d_\varepsilon^{(b)}<d_\varepsilon^{(a)}\}}\right]
&=\E_a\left[1_{\{a<X_\varepsilon< b\}}  \E_{X_\varepsilon}\left[ e^{-(\gamma+\lambda) T_b}1_{\{T_b<T_a\}}\right]\right]\\
&\equi_{\varepsilon\downarrow 0} \frac{1}{s^\prime(a)} \frac{\text{d}}{\text{d}x} \E_{x}\left[ e^{-(\gamma+\lambda) T_b}1_{\{T_b<T_a\}}\right]_{\big|x=a^+}  \times\int_0^{\varepsilon}p(t;a,a) dt.
\end{align*}

\noindent
Similarly,  the denominator is equivalent to :
$$
\frac{1}{s^\prime(a)}\left( \frac{\text{d}}{\text{d}x} \E_{x}\left[ e^{-(\gamma+\lambda) T_a}1_{\{T_a<u\}}\right]_{\big|x=a^-}  -  \frac{\text{d}}{\text{d}x} \E_{x}\left[e^{-(\gamma+\lambda) T_a}1_{\{T_a<T_b\}}\right]_{\big|x=a^+}\right)     \times\int_0^{\varepsilon}p(t; a,a) dt.
$$
Finally, from (\ref{eq:e/p}) and (\ref{eq:KW}), the last term is negligible since
$$\frac{\Delta_a(\varepsilon)}{\int_0^{\varepsilon} p(t; a,a ) dt} \leq  \max_{z\in \R}(\alpha(z)\beta(z)) \frac{\Pb_a(T_b\leq \varepsilon)}{\int_0^{\varepsilon} p(t; a,a ) dt} \xrightarrow[\varepsilon\downarrow0]{} 0$$
and we obtain, after some simplifications, 
\begin{align}
\notag F_{\alpha, \beta, \gamma,\lambda }(a) &= \frac{ \frac{\text{d}}{\text{d}x} \E_{x}\left[ e^{-(\gamma+\lambda) T_b}1_{\{T_b<T_a\}}\right]_{\big|x=a^+}  }{  \frac{\text{d}}{\text{d}x} \E_{x}\left[ e^{-(\gamma+\lambda) T_a}1_{\{T_a<u\}}\right]_{\big|x=a^-}  -  \frac{\text{d}}{\text{d}x} \E_{x}\left[e^{-(\gamma+\lambda) T_a}1_{\{T_a<T_b\}}\right]_{\big|x=a^+} }F_{\alpha, \beta, \gamma,\lambda }(b)\\
\notag &=\frac{ \frac{W_{\gamma+\lambda}^{(a)\prime}(a)}{W_{\gamma+\lambda}^{(a)}(b)}}{
     \frac{\Phi_{\gamma+\lambda, -}^{\prime}(a)}{\Phi_{\gamma+\lambda, -}(a)} +s^\prime(a)\int_u^{+\infty}e^{-(\gamma+\lambda) t}\nu^{(a)}_-(dt)-\frac{W_{\gamma+\lambda}^{(b)\prime}(a)}{W_{\gamma+\lambda}^{(b)}(a)}
}F_{\alpha, \beta, \gamma,\lambda }(b)  \\
\label{eq:Fa=Fb}&= \frac{1}{\Psi^{(-)}_{\gamma+\lambda}(b,a,u)}F_{\alpha, \beta, \gamma,\lambda }(b).
\end{align}

\begin{remark}
By identification, the Laplace transform of $T_b$ on the event $\{T_b < \kappa_u^{(a,-)}\}$ is given by  
$$\E_a\left[e^{-\lambda T_b}  1_{\{T_b < \kappa_u^{(a,-)}\}}\right] =  \frac{1}{\Psi^{(-)}_{\lambda}(b,a,u)}.$$
\end{remark}

\subsection{Study of $F_{\alpha, \beta, \gamma,\lambda }(b)$}

\noindent
First, if $X_\varepsilon>b$, we need to take care as whether or not the excursion above $b$ straddling $\varepsilon$ is longer than $v$. On the one hand, if this excursion is longer than $v$, then $\kappa_{v}^{(b,+)} = g_{\varepsilon}^{(b)}+v$, hence, applying as before the Markov property at the time $d_\varepsilon^{(b)}$~: 
\begin{align*}
&\E_b\left[e^{-\gamma \kappa_u^{(a,-)}- \lambda \kappa_v^{(b,+)}} \alpha \big(X_{ \kappa_u^{(a,-)}}\big) \beta \big(X_{ \kappa_v^{(b,+)}}\big)1_{\{\kappa_v^{(b,+)}\leq \kappa_u^{(a,-)}<+\infty\}}1_{\{X_\varepsilon\geq b\}} 1_{\{d_\varepsilon^{(b)}\geq v + g_\varepsilon^{(b)}\}}  \right] \\
&=\E_b\left[e^{ -\gamma d_\varepsilon^{(b)}-\lambda (g_\varepsilon^{(b)}+v)} \beta \big(X_{g_\varepsilon^{(b)}+v}   \big)   1_{\{X_\varepsilon\geq b\}} 1_{\{d_\varepsilon^{(b)}\geq v + g_\varepsilon^{(b)}\}}  \right]  \E_b\left[e^{-\gamma \kappa_u^{(a,-)}}\alpha\big( X_{ \kappa_u^{(a,-)}} \big)1_{\{ \kappa_u^{(a,-)}<+\infty\}} \right] .
\end{align*}
On the other hand, if this excursion is shorter than $v$, then 
\begin{multline*}
\E_b\left[e^{-\gamma \kappa_u^{(a,-)}- \lambda \kappa_v^{(b,+)} } \alpha \big(X_{ \kappa_u^{(a,-)}}\big) \beta \big(X_{ \kappa_v^{(b,+)}}\big)1_{\{\kappa_v^{(b,+)}\leq \kappa_u^{(a,-)}<+\infty\}}1_{\{X_\varepsilon\geq b\}} 1_{\{d_\varepsilon^{(b)}< v + g_\varepsilon^{(b)}\}}  \right] \\
=\E_b\left[   e^{-(\gamma+\lambda)d_\varepsilon^{(b)}} 1_{\{X_\varepsilon\geq b\}} 1_{\{d_\varepsilon^{(b)}<v + g_\varepsilon^{(b)}\}}  \right] F_{\alpha, \beta, \gamma,\lambda }(b).
\end{multline*}

\noindent
Next, if $a<X_\varepsilon<b$, we obtain, applying the Markov property at the times $d_\varepsilon^{(a)}$ and $d_\varepsilon^{(b)}$ :
\begin{multline}
\label{Fab}\E_b\left[e^{- \lambda \kappa_v^{(b,+)}-\gamma \kappa_u^{(a,-)} } \alpha \big(X_{ \kappa_u^{(a,-)}}\big) \beta \big(X_{ \kappa_v^{(b,+)}}\big)1_{\{\kappa_v^{(b,+)}\leq \kappa_u^{(a,-)}<+\infty\}}1_{\{a<X_\varepsilon< b\}}\right]\\=\E_b\left[e^{-(\gamma+\lambda) d_\varepsilon^{(b)}}1_{\{a<X_\varepsilon< b\}}1_{\{d_\varepsilon^{(b)}<d_\varepsilon^{(a)}\}}\right]F_{\alpha, \beta, \gamma,\lambda }(b) \\+\E_b\left[e^{-(\gamma+\lambda) d_\varepsilon^{(a)}}1_{\{a<X_\varepsilon< b\}}1_{\{d_\varepsilon^{(a)}<d_\varepsilon^{(b)}\}}\right]F_{\alpha, \beta, \gamma,\lambda }(a).
\end{multline}

\noindent
Finally, the last term when $X_\varepsilon\leq a$ is again negligible :
$$\Delta_b(\varepsilon)=\E_b\left[e^{-\gamma \kappa_u^{(a,-)}- \lambda \kappa_v^{(b,+)} } \alpha \big(X_{ \kappa_u^{(a,-)}}\big) \beta \big(X_{ \kappa_v^{(b,+)}}\big)1_{\{\kappa_v^{(b,+)}\leq \kappa_u^{(a,-)}<+\infty\}}1_{\{X_\varepsilon\leq a\}}\right].$$

\noindent
Passing to the limit as $\varepsilon\downarrow0$ and proceeding as in Proposition \ref{prop:1}, we deduce that 
\begin{multline*}
\Psi^{(+)}_{\gamma+\lambda}(a,b,v)F_{\alpha, \beta, \gamma,\lambda }(b)  - F_{\alpha, \beta, \gamma,\lambda }(a)\\
= e^{-\lambda v-\gamma v}\frac{W_{\gamma+\lambda}^{(b)}(a)}{\omega_{\gamma+\lambda}}  \E_b\left[e^{-\gamma \kappa_u^{(a,-)}} \alpha\big( X_{ \kappa_u^{(a,-)}} \big)1_{\{ \kappa_u^{(a,-)}<+\infty\}} \right]\E_{b}^{\uparrow,v}\left[    \frac{\Phi_{\gamma,+}(X_v)}{\Phi_{\gamma,+}(b)}\beta\left( X_v\right)\right]\nu^{(b)}_+[v,+\infty]
\end{multline*}
and the expression for $F_{\alpha, \beta, \gamma,\lambda }(b)$ follows by plugging (\ref{eq:Fa=Fb}) in this last formula. Finally, to get the general expression of Theorem \ref{theo:1} for $a\leq x \leq b$, one simply need to apply the Markov property
$$F_{\alpha, \beta, \gamma,\lambda }(x) = \E_x\left[e^{-(\gamma+\lambda) T_a }1_{\{T_a<T_b\}}\right]F_{\alpha, \beta, \gamma,\lambda }(a) +  \E_x\left[e^{-(\gamma+\lambda) T_b }1_{\{T_b<T_a\}}\right]F_{\alpha, \beta, \gamma,\lambda }(b)
$$
and then use (\ref{eq:Fa=Fb}).
 \qed

\section{Applications}\label{sec:4}

We briefly give a few examples where the expressions for the L\'evy measures $\nu^{(a)}_-$ and $\nu^{(b)}_+$ may be written down explicitly.

\subsection{Brownian motion with drift $\mu$}

Assume that $X$ is a Brownian motion with drift $\mu$. 
Its scale function may be chosen such that  :
$$\displaystyle s^\prime(x) = e^{-2\mu x}.$$
The two eigenfunctions $\Phi_{\lambda,+}$ and $\Phi_{\lambda,-}$ are given by :
$$\Phi_{\lambda,+}(x) = \exp\left(- x \left(\sqrt{2\lambda +\mu^2}+\mu\right)\right) \quad\text{ and }\quad 
\Phi_{\lambda,-}(x) =\exp\left( x \left(\sqrt{2\lambda +\mu^2}-\mu\right)\right)$$
and
$$\Pb_x(T_a\in dt) = \frac{|a-x|}{\sqrt{2\pi} t^{3/2}}\exp\left(-\frac{(a-x-\mu t)^2}{2t}\right) dt.$$
As a consequence, the L\'evy measures $\nu^{(a)}_\pm$ only depend on $a$ through the scale function $s$ :
$$\nu^{(a)}_\pm(dt) = e^{2\mu a}  \frac{1}{\sqrt{2\pi} t^{3/2}} e^{-\frac{\mu^2}{2}t }dt,$$
and
$$\nu^{(a)}_\pm(\{+\infty\}) = e^{2\mu a} \left(|\mu| \pm \mu \right).$$
When $\mu=0$, one may choose $s(x)=x$ and applying Remark \ref{rem:rec}, we recover for instance the formula for a standard Brownian motion,  i.e. for $a\leq x\leq b$   (see  \cite{DW2}) :
$$\Pb_x\left(\kappa_v^{(b,+)} \leq \kappa_u^{(a,-)}\right) =   \frac{\sqrt{u} + (x-a)\sqrt{\frac{2}{\pi}}}{\sqrt{v}+\sqrt{u}+ (b-a)\sqrt{\frac{2}{\pi}}}.$$

\subsection{Three-dimensional Bessel process with drift $\mu>0$ in the wide sense}

Let us consider the solution of the following SDE :
$$X_t = x + B_t + \mu \int_0^{t} \frac{\cosh(\mu X_s)}{\sinh(\mu X_s)}ds$$
where $x>0$ and $B$ is a standard Brownian motion. The law of $X$ is that of a Brownian motion with drift $\mu>0$ conditioned to stay positive, or alternatively that of a three-dimensional Bessel process with drift $\mu$ in the wide sense, see Pitman \& Yor \cite{PY2} or Watanabe \cite{Wa}. Its scale function may be chosen such that 
$$s^\prime(x) = \frac{\mu^2}{\sinh^2(\mu x)}$$
and the two eigenfunctions are given by 
$$\Phi_{\lambda,+}(x)=\mu \frac{ \exp\left(-x\sqrt{2\lambda+\mu^2}\right)}{\sinh(\mu x)} \quad \text{ and }\quad    \Phi_{\lambda,-}(x) = \mu \frac{\sinh(x\sqrt{2\lambda+\mu^2})}{\sinh(\mu x)}. $$
As a consequence, we deduce from \cite[p.258]{Erd} that  
$$\Pb_x(T_a\in dt)/dt=
\begin{cases}
 \displaystyle \frac{\sinh(\mu a )}{\sinh(\mu x)} \frac{1}{\sqrt{2\pi} t^{3/2}} \sum_{n=-\infty}^{+\infty} ((2n+1)a -x) e^{-\frac{\left((2n+1)a-x\right)^2}{2t} -\frac{\mu^2}{2}t }& \quad \text{if } x< a,\\
& \\
&\\
\displaystyle \frac{\sinh(\mu a )}{\sinh(\mu x)} \frac{(x-a)}{\sqrt{2\pi} t^{3/2}} \exp\left(-\frac{(x-a)^2}{2t}-\frac{\mu^2}{2}t\right)  &  \quad \text{if } x> a,
\end{cases}
$$
hence  we obtain
$$
\nu^{(a)}_-(dt) = \frac{\sinh^2(\mu a)}{\mu^2} \frac{1}{\sqrt{2\pi}t^{3/2}}e^{-\frac{\mu^2}{2}t} \sum_{n=-\infty}^{+\infty} \left(\frac{\mu\cosh(\mu a)}{\sinh(\mu a)}  2na + 1 - \frac{4n^2a^2}{t} \right) \, e^{-\frac{2n^2 a^2}{t}}dt$$
and
$$\nu^{(b)}_+(dt) =  \frac{\sinh^2(\mu b)}{\mu^2} \frac{1}{\sqrt{2\pi}t^{3/2}}e^{-\frac{\mu^2}{2}t}dt.$$
Note that the measure $\nu^{(b)}_+$ admits an atom at $+\infty$ which is given by :
$$\nu^{(b)}_+(\{+\infty\}) = \frac{e^{2\mu b}-1}{2\mu}.$$
As a consequence, we deduce for instance that :
$$\Pb_a\left(\kappa_u^{(a,-)}<+\infty\right) = \frac{\nu_-^{(a)}[u,+\infty)}{ \frac{e^{2\mu a}-1}{2\mu}  +\nu_-^{(a)}[u,+\infty) }.$$

\subsection{Reflected Brownian motion}
Assume that $X=|B|$ with $B$ a standard Brownian motion. The two eigenfunctions $\Phi_{\lambda,+}$ and $\Phi_{\lambda,-}$ are given by :
$$\Phi_{\lambda,+}(x) = \exp\left(- x \sqrt{2\lambda}\right) \quad\text{ and }\quad 
\Phi_{\lambda,-}(x) =\cosh\left( x\sqrt{2\lambda }\right)$$
and from Borodin \& Salminen \cite{BS}, the distribution of $T_a$ admits the density
$$\Pb_x(T_a\in dt)/dt=
\begin{cases}
 \displaystyle  \frac{1}{\sqrt{2\pi} t^{3/2}} \sum_{n=-\infty}^{+\infty} (-1)^n ((2n+1)a +x) e^{-\frac{\left((2n+1)a+x\right)^2}{2t} }& \quad \text{if } x< a,\\
& \\
&\\
\displaystyle \frac{(x-a)}{\sqrt{2\pi} t^{3/2}} \exp\left(-\frac{(x-a)^2}{2t}\right)  &  \quad \text{if } x> a,
\end{cases}
$$
Therefore, we deduce that 
$$
\nu^{(a)}_-(dt) =  \frac{1}{\sqrt{2\pi}t^{3/2}} \sum_{n=-\infty}^{+\infty}(-1)^{n+1} \left( \frac{4 a^2 n^2}{t}-1\right) \, e^{-\frac{2a^2 n^2}{t}}dt$$
and
$$\nu^{(b)}_+(dt) = \frac{1}{\sqrt{2\pi}t^{3/2}}dt.$$

\addcontentsline{toc}{section}{References}

\begin{thebibliography}{10}

\bibitem{AW}
J. H. M. Anderluh and J. A. M. van der Weide. Double-sided Parisian option pricing. {\em Finance Stoch. } {\bf 13}, no. 2, 205--238, 2009.

\bibitem{BS}
A.N. Borodin and P. Salminen. Handbook of Brownian motion - facts and formulae. Second edition. Probability and its Applications. {\em Birkh\"auser Verlag, Basel}, 2002

\bibitem{CJY}
M. Chesney, M. Jeanblanc-Picqu\'e and  M. Yor. Brownian excursions and Parisian barrier options. {\em Adv. in Appl. Probab.} {\bf  29}, no. 1, 165--184, 1997.

\bibitem{DW1}
A. Dassios and S. Wu. Perturbed Brownian motion and its application to Parisian option pricing. {\em Finance and Stochastics} {\bf 14}, 473--494,  2010.

\bibitem{DW2}
A. Dassios and S. Wu. Double-barrier Parisian options. {\em J. Appl. Probab.} {\bf  48}, no. 1, 1--20, 2011.


\bibitem{Erd}
A. Erd\'elyi, W. Magnus, F. Oberhettinger and F.G. Tricomi.  Tables of integral transforms. Vol. I. {\em McGraw-Hill Book Company, Inc., New York-Toronto-London,} 1954.

\bibitem{Ge}
R. K. Getoor. Excursions of a Markov process. {\em Ann. Probab.} {\bf  7}, no. 2, 244--266,  1979.

\bibitem{IMK}
K. It\^o and H. P. McKean. Diffusion processes and their sample paths. Second printing, corrected. Die Grundlehren der mathematischen Wissenschaften, Band 125. {\em Springer-Verlag, Berlin-New York,} 1974.

\bibitem{KW}
S. Kotani and S. Watanabe.  Krein's spectral theory of strings and generalized diffusion processes. {\em Functional analysis in Markov processes}.  Lecture Notes in Math., {\bf 923}, {\em Springer, Berlin-New York},  pp. 235--259, 1982. 


\bibitem{LL}
C. Labart and J. Lelong. Pricing double barrier Parisian options using Laplace transforms. {\em Int. J. Theor. Appl. Finance} {\bf 12}, no. 1, 19--44, 2009.

\bibitem{LCZ}
R. Loeffen, I. Czarna  and Z. Palmowski. Parisian ruin probability for spectrally negative L\'evy processes. {\em Bernoulli } {\bf 19},  no. 2, 599--609, 2013.


\bibitem{PY}
J. Pitman and M. Yor. Hitting, occupation and inverse local times of one-dimensional diffusions: martingale and excursion approaches. {\em Bernoulli} {\bf 9}, no. 1, 1--24, 2003.


\bibitem{PY2} 
J. Pitman and M. Yor. Bessel processes and infinitely divisible laws. {\em Stochastic integrals}, pp. 285--370, Lecture Notes in Math., {\bf 851}, Springer, Berlin, 1981. 

\bibitem{RY}
D. Revuz and M. Yor. Continuous martingales and Brownian motion. Third edition. Grundlehren der Mathematischen Wissenschaften [Fundamental Principles of Mathematical Sciences], 293. {\em Springer-Verlag, Berlin,} 1999.

\bibitem{Wa} 
S. Watanabe. On time inversion of one-dimensional diffusion processes. {\em Z. Wahrscheinlichkeitstheorie und Verw. Gebiete } {\bf 31}, 115--124, 1974/1975.

\end{thebibliography}

\end{document}